\documentclass[a4paper,10pt,reqno]{elsarticle-mr}

\usepackage{color}						%Farbe
\definecolor{egreen}{rgb}{0,0.6,0}

\usepackage[%
   colorlinks=true,
   urlcolor=blue,       % \href{...}{...} external (URL)
   filecolor=green,     % \href{...} local file
   citecolor=green,     % \cite{...} 
   linkcolor=red,       % \ref{...} and \pageref{...}
   bookmarks=true,
   unicode,
   plainpages=false,
   ]{hyperref}
\usepackage[active]{srcltx}

\usepackage{amsmath,color}
\usepackage{amssymb,amsthm,amsxtra,amscd}

\numberwithin{equation}{section}

\newtheorem{theorem}{Theorem}[section]

\newtheorem{lemma}[theorem]{Lemma}
\newtheorem{proposition}[theorem]{Proposition}

\newtheorem{ass}[theorem]{Assumption}

\newcommand{\abs}[1]{\lvert#1\rvert} 		   % Absolute value notation
\newcommand\norm[1]{\left\lVert#1\right\rVert} % Norm notation
	
\newcommand{\fdg}{{\,\big|\,}}

\begin{document}

%%%%%%%%%%%%%%%%%%

\begin{frontmatter}

\title{Note on the existence theory for
  evolution equations with pseudo-monotone operators}
%[Evolution equations with pseudo-monotone operators]

\author[eb]{E.~B{\"a}umle}
\ead{Erik.Baeumle@gmx.de}

\author[mr]{M.~R\r u\v zi\v cka \corref{cor1}}
\ead{rose@mathematik.uni-freiburg.de}

\cortext[cor1]{Corresponding author}

\address[eb]{Bl{\"u}tenweg 12, D-77656 Offenburg, GERMANY} 

\address[mr]{Institute of Applied Mathematics,
  Albert-Ludwigs-University Freiburg, Eckerstr.~1, D-79104 Freiburg,
  GERMANY.}

\begin{abstract}
In this note we present a framework which allows to prove an abstract
existence result for  evolution equations with pseudo-monotone
operators. The assumptions on the spaces and the operators can be
easily verified in concrete examples. 
\end{abstract}


% \begin{classification}
% 47 H05, 35 K90, 35 A01. 
% \end{classification}

\end{frontmatter}

%%%%%%%%%%%%%%
%\maketitle

\section{Introduction}
The theory of pseudo-monotone operators is very useful in proving the
existence of solutions of non-linear problems. The main theorem on
pseudo-monotone operators, due to Brezis \cite{bre68}, shows the
surjectivity of a pseudo-monotone, bounded, coercive operator.  This
result extends the fundamental contribution of Browder~\cite {bro63}
and Minty \cite{min63} on monotone operators to pseudo-monotone
operators. The prototype of such an operator is a sum of a monotone
operator and a compact operator. A huge class of elliptic partial
differential equations can be treated in this framework, since many "{}lower
order terms''{} define a compact operator due to compact embedding
theorems. 

The theory of monotone operators can easily be generalized to the
treatment of non-linear evolution equations (cf.~\cite{lions-quel},
\cite{ZeidB}, \cite{show97}). In the fundamental contribution
\cite{lions-quel} Lions combines, among others, monotonicity methods
with compactness methods.  Even though, there exists a general
existence result for evolutionary pseudo-monotone, coercive, bounded
operators (cf.~\cite{show97}, \cite{rose-book}, \cite{Roub}), its
applicability to concrete problems is limited. This is due to the fact
that the treatment of "{}lower order terms''{} as a compact operator
needs usually additional information on the time derivative.  The
incorporation of the time derivative into the function space however
contradicts the required coercivity of the operator.  The way out of
this problem for evolution problems is to repeat and adapt the
arguments given in \cite{lions-quel} to the concrete application to be
treated. This is a non-satisfactory situation. There are many
contributions to develop a general existence theory for evolution
equations with pseudo-monotone operators (cf.~\cite{Hir1},
\cite{Hir2}, \cite{Shi}, \cite{show97}, \cite{Roub},
\cite{liu11}). 

The purpose of this note is to provide an existence theory for
evolution equations with pseudo-monotone operators which is easily
applicable. To this end we use and extend ideas from \cite{Shi} and
\cite{Roub}. Essentially, one has to check whether the operator is on
almost all time slices pseudo-monotone, coercive and satisfies certain
natural growth conditions. This enables us to prove a generalization
of Hirano's lemma, which in turn allows the limiting process in the
Galerkin approximation of the evolution problem. To verify the
assumptions of Hirano's lemma we need a technical assumption on the
function spaces, which can be traced back to
\cite{lions-quel}. Note that such an assumption is not present in
\cite{Shi}. However, we are not able to follow one argument
in the proof of Lemma 3 there, which would circumvent this
technical assumption. In concrete application the technical assumption
is easily verified.

The paper is organized as follows: In Section 2 we introduce the
notation and collect some basic results for pseudo-monotone operators
and evolution problems. In Section 3 we give the assumption on the
function spaces and the operator and formulate the main
theorem. Section 4 is devoted to the proof of the generalization of
Hirano's lemma. The main theorem is proved in Section 5. Finally we
apply the main result in Section 6 to treat the evolution $p$-Laplace
equation with a lower order term and the evolution equation for
generalized Newtonian fluids.

\section{Preliminaries}
\subsection{Notation and conventions}
%<<<<<<< HEAD
%For a real Banach space $V$ we denote by $V^\ast$ its dual space and $V^\ast$ will be equipped with the norm
%$\norm{f}_{V^\ast}= \sup_{x \in V, \norm{x}_V\leq 1} \langle f,x\rangle_V$, where $\langle f,x\rangle_V$ 
%is the duality pairing. Integrals will usually be written
%as $\int_\Omega f(t)$
%instead of $\int_\Omega f(t)\, dt$. $c$ will denote a generic constant which may change from line to line.
%With $C(0,T;V)$ we will denote the space of continuous functions on $[0,T]$ with values in the Banach space $V$.
%Finally we use the standard notation for Bochner spaces, see \cite{Evans}.
%=======
% For a real Banach space $V$ we denote by $V^\ast$ its dual space and $V^\ast$ will be equipped with the norm
% $\norm{f}_{V^\ast}= \sup_{x \in V, \norm{x}_V\leq 1} \langle f,x\rangle_V$, where $\langle f,x\rangle_V$ 
% is the duality pairing. Integrals will usually be written
For a Banach space $X$ with norm $\Vert \cdot \Vert _X$ we denote by
$X^\ast$ its dual space equipped with the norm
$\Vert \cdot \Vert_{X^\ast}$. The duality pairing is denoted by
$\langle \cdot,\cdot\rangle_X$. All occurring Banach spaces are
assumed to be real. By an embedding we always understand a continuous
embedding. Time integrals will usually be written as $\int_0^T f(t)$
instead of $\int_0^T f(t)\, dt$. By $c$ we denote a generic constant
which may change from line to line. Finally we use the standard
notation for Bochner spaces (cf.~\cite{Evans}).
%>>>>>>> rose/master
\subsection{Auxiliary results}
From now on $V$ always denotes a separable, reflexive Banach space and
$H$ a Hilbert space. If the embedding $V\hookrightarrow H$ is dense,
we call $(V,H,V^\ast)$ a Gelfand-Triple. Using the Riesz
representation theorem we obtain
$V\hookrightarrow H \cong H^\ast \hookrightarrow V^\ast$ where both
embeddings are dense.  An operator $A\colon V\to V^*$ is said to be
monotone if $\langle Ax-Ay,x-y\rangle_V \geq 0$ for all $x,y \in V$.
The operator ${A \colon V\to V^*}$ is said to be pseudo-monotone if
$x_n \rightharpoonup x$ in V and
$\limsup_{n\rightarrow \infty} \langle Ax_n,x_n-x \rangle_V\leq 0$
implies
\begin{align*}
  \langle Ax,x-y \rangle_V \leq \liminf\limits_{n\rightarrow \infty} \,
  \langle Ax_n, x_n -y \rangle_V \qquad \text{for all } y \text{ in }
  V. 
\end{align*}  
The following proposition introduces some typical types of pseudo-monotone operators.
\begin{proposition}\label{pseudo}
Let $A,B:V\rightarrow V^\ast$ be operators.  Then there holds:
\begin{itemize}
\item[(a)] If $A$ is monotone and hemicontinuous, then $A$ is pseudo-monotone.
\item[(b)] If $A$ is strongly continuous, then $A$ is pseudo-monotone.
\item[(c)] If $A$ and $B$ are pseudo-monotone, then $A+B$ is pseudo-monotone.
\end{itemize}
\end{proposition}   
\begin{proof}
See \cite[Proposition 27.6]{ZeidB}
\end{proof}
Next we will state some well known results concerning Bochner and
Bochner--Sobolev spaces, which will be used in the following. 
\begin{theorem}[Pettis]\label{pettis}
%<<<<<<< HEAD
%If $V$ is separable, then $u$ is Bochner-measurable if and only if is weakly measurable in the sense: $\langle v^\ast ,u(\cdot) \rangle_V$ is Lebesgue-measurable for any $v^\ast \in V^\ast$.  
%=======
A function $u \colon (0,T) \to V$  is Bochner measurable if and only
if $u$ is weakly measurable in the sense: $\langle v^\ast ,u(\cdot)
\rangle_V$ is Lebesgue measurable for any $v^\ast \in V^\ast$.   
%>>>>>>> rose/master
\end{theorem}
\begin{proof}
See \cite[Theorem 1.34]{Roub}.
\end{proof}
%\begin{proposition}
%Let $V$ be a reflexive and separable Banach space and let $1<p<\infty$, $\frac{1}{p} +\frac{1}{p'} =1$. Then the following hold:
%\begin{itemize}
%\item[(a)] To each function $v \in L^{p'}(0,T;V^\ast)$ there corresponds a unique functional $\overline{v} \in L^p(0,T;V)^\ast$ with $ \langle \overline{v}, u \rangle_{L^p(0,T;V)}=\int_0^T \langle v(t),u(t) \rangle_{V}$ for all $u \in L^p(0,T;V)$.
%\item[(b)] Conversely, to each $\overline{v} \in L^p(0,T;V)^\ast$ there exists exactly one  $v \in L^{p'}(0,T;V^\ast)$ with $ \langle \overline{v}, u \rangle_{L^p(0,T;V)}=\int_0^T \langle v(t),u(t) \rangle_{V}$ for all $u \in L^p(0,T;V)$.\\
%\item[(c)] The Banach space $L^p(0,T;V)$ is reflexive and separable.
%\end{itemize}
%\end{proposition}
%\begin{proof}
%See \cite[Proposition 23.7]{ZeidA}
%\end{proof}
%<<<<<<< HEAD
%Let $V,W$ be Banach spaces such that $V\hookrightarrow W$ and $1\leq p,q\leq\infty$. 
%We say that a function $u \in L^p(0,T;V)$ has a generalized derivative in $L^q(0,T;W)$ if there exists a function $g \in L^q(0,T;W)$ such that
%\begin{align*}
%\int_0^T \varphi'(t) u(t) = - \int_0^T \varphi(t)g(t),\quad \textnormal{for all } \varphi \in C_0^\infty((0,T)).
%\end{align*}
%If such a function $g$ exists, it is unique and we denote
%=======
Let $W$ be a Banach space such that the embedding $V\hookrightarrow W$
is dense and let $1<p,q<\infty$.  We say that a
function $u \in L^p(0,T;V)$ has a generalized derivative in
$L^q(0,T;W)$ if there exists a function $g \in L^q(0,T;W)$ such that
%>>>>>>> rose/master
\begin{align*}
  \int_0^T \varphi'(t) u(t) = - \int_0^T \varphi(t)g(t),\quad
  \textnormal{for all } \varphi \in C_0^\infty((0,T)). 
\end{align*}
If such a function $g$ exists, it is unique and we set
$\frac{du}{dt}: =g$.  With this definition of generalized derivatives
we are able to introduce Bochner--Sobolev spaces. For $1<p,q<\infty$
we define
\begin{align*}
  W^{1,p,q}(0,T;V,W):= \left\lbrace u\in L^p(0,T;V) \mid \frac{du}{dt} \in L^q(0,T;W) \right\rbrace.
\end{align*}
With the norm 
\begin{align*}
\norm{u}_{W^{1,p,q}(0,T;V,W)}:= \norm{u}_{L^p(0,T;V)} + \norm{\frac{du}{dt}}_{L^q(0,T;W) }
\end{align*}
this space is a reflexive Banach space.
\begin{proposition}\label{Einbettung1}
  Any element $u \in {W^{1,p,q}(0,T;V,W)}$ (defined almost everywhere)
  possesses a continuous representation on $[0,T]$ with values in $W$,
  and $W^{1,p,q}(0,T;V,W)$ embeds into $C(0,T;W)$.
\end{proposition}
\begin{proof}
See \cite[Proposition II.5.11]{BF}.
\end{proof}
Let $(V,H,V^\ast)$ be a Gelfand-Triple and $1<p<\infty$. Then we define the  Bochner--Sobolev space
\begin{align*}
W_p^1(0,T;V,H):= \left\lbrace u\in L^p(0,T;V) \fdg \frac{du}{dt} \in L^{p'}(0,T;V^\ast) \right\rbrace,
\end{align*}
where $\frac{1}{p}+\frac{1}{p'}=1$. With the norm 
\begin{align*}
\norm{u}_{W_p^1(0,T;V,H)}:= \norm{u}_{L^p(0,T;V)} + \norm{\frac{du}{dt}}_{L^{p'}(0,T;V^\ast) }
\end{align*}
this space is a reflexive Banach space. Since the embedding 
$V\hookrightarrow V^\ast$ is dense, Proposition \ref{Einbettung1} is
valid. In this setting we can sharpen Proposition \ref{Einbettung1} as
follows: 
\begin{proposition}\label{Einbettung2}
  Let $(V,H,V^\ast)$ be a Gelfand-Triple. Then $W_p^1(0,T;V,H)$ embeds
  into $C(0,T;H)$. Moreover, the integration by parts formula
\begin{align*}
  \left( u(t) ,v(t)\right)_H- \left( u(s) ,v(s)\right)_H = \int_s^t
  \left\langle \frac{du}{dt}(\tau),v(\tau)\right\rangle_V
  +\left\langle \frac{dv}{dt}(\tau),u(\tau)\right\rangle_V 
\end{align*}
holds for any $u,v \in W_p^1(0,T;V,H)$ and arbitrary $0\leq s,t\leq T$.
\end{proposition}
\begin{proof}
See \cite[Proposition 23.23]{ZeidA}.
\end{proof}
\begin{proposition}\label{ExistenzZeitab}
  Let $(V,H,V^\ast)$ be a Gelfand-Triple and let
  $1< p,q< \infty$. The function $u \in L^p(0,T;V)$ possesses a
  generalized derivative $\frac{du}{dt} \in L^q(0,T;V^\ast)$ if there
  is a function $w \in L^q(0,T;V^\ast)$ such that
\begin{align*}
  \int_0^T \left(u(t),v   \right)_H\varphi'(t) = -
  \int_0^T\left\langle w(t) ,v\right\rangle_V\varphi(t) 
\end{align*}
for all $v \in V$ and all $\varphi \in C_0^\infty((0,T)).$ Then $\frac{du}{dt} =w$.
\end{proposition}
\begin{proof}
See \cite[Proposition 23.20]{ZeidA}.
\end{proof}

\section{Statement of the main theorem}
Now we formulate the assumptions on the spaces and the operator which
allow us to proof an abstract existence result for an evolution
equation with pseudo-monotone operators. 
\begin{ass}[Spaces]\label{vorraeume}
  Let $(V,H,V^\ast)$ be a Gelfand-Triple. Assume that there exists a
  reflexive, separable Banach space $Z$ such that the embedding
  ${Z\hookrightarrow V}$ is dense. Moreover, assume that there exists an
  increasing sequence of finite dimensional subspaces
  $V_n \subseteq Z$, such that $\cup _{n \in \mathbb{N}} V_n$ is dense
  in $V$.  Additionally, assume that there exists self-adjoint projections
  $P_n:H \rightarrow H$, such that $P_n(V)=V_n$ and
  $\norm{P_{n\ |Z}}_{\mathcal{L}(Z,Z)}\leq c$ with a constant $c$
  independent of $n\in\mathbb{N}$.
\end{ass}
\begin{ass}[Operators]\label{voroperatorfamilie}
Let $\lbrace A(t) \mid 0\leq t\leq T\rbrace$ be a family of operators from $V$ to $V^\ast$ with the following properties:
\begin{itemize}
\item[\textnormal{(A1)}] $A(t):V \rightarrow V^\ast$ is pseudo-monotone for almost every $t \in [0,T]$.
%<<<<<<< HEAD
%\item[\textnormal{(A2)}] For every $u \in L^p(0,T;V)\cap L^\infty(0,T;H)$ the mapping $t \mapsto A(t)u(t)$ 
%from $[0,T]$ to $V^\ast$ is Bochner-measurable.
%\item[\textnormal{(A3)}] There exists a positive constant $C_1$ and a non-negative function\\ $C_2 \in L^1((0,T))$,
%such that
%=======
\item[\textnormal{(A2)}] For every
  $u \in L^p(0,T;V)\cap L^\infty(0,T;H)$ the mapping
  $t \mapsto A(t)u(t)$ from $[0,T]$ to $V^\ast$ is Bochner-measurable.
\item[\textnormal{(A3)}] There exists a positive constant $c_1$ and a
  non-negative function \\$C_2 \in L^1((0,T))$, such that
%>>>>>>> rose/master
\begin{align*}
\langle A(t)x,x  \rangle_V \geq c_1 \norm{x}_V^{p} -C_2(t)
\end{align*}
for almost every $t \in [0,T]$ and all $x\in V$.
\item[\textnormal{(A4)}] There exists $0\le q<\infty$, as well as
  constants $c_3>0$, $c_4\geq 0$ and a non-negative function
  $C_5\in L^{p'}((0,T))$, such that
\begin{align*}
\norm{A(t)x}_{V^\ast}\leq c_3\norm{x}_V^{p-1}+ c_4 \norm{x}_H^q \norm{x}_V^{p-1} +  C_5(t)
\end{align*}
for almost every $t \in [0,T]$ and all $x\in V$.
\end{itemize}
\end{ass}
\begin{theorem}[Main Theorem]\label{maintheorem}
  If the Assumptions \ref{vorraeume} and \ref{voroperatorfamilie} are
  satisfied, then for every $u_0 \in H$, $f \in L^{p'}(0,T;V^\ast)$
  there exists a function $u \in W_p^1(0,T;V,H)$ such that
\begin{align*}
\frac{du}{dt}(t) +A(t)u(t)=f(t) &\qquad \text{in }V^\ast \text{ for a.e.} \ t \in [0,T]\\
u(0)=u_0 &\qquad  \text{in}\ H.
\end{align*}
\end{theorem}

\section{Hirano's Lemma}
In this section we prove a generalized version of Hirano's Lemma
(cf. \cite{Hir1}, \cite{Hir2}, \cite{Shi}). First we show that the
induced operator is bounded between the correct function spaces.
\begin{lemma}\label{induzierterOp}
  Assume that $\lbrace A(t) \mid 0\leq t\leq T\rbrace$ satisfy 
  Assumption \ref{voroperatorfamilie}. Then the in\-duced operator
  $(\mathcal{A}u)(t):=A(t)u(t)$,
  $\mathcal{A}\colon L^p(0,T;V)\cap L^\infty(0,T;H) \!\rightarrow\!
  L^{p'}\!(0,T;V^\ast)$
is bounded.
\end{lemma}
\begin{proof}
The Bochner-measurability holds due to (A2). With the growth condition (A4) we conclude
\begin{align*}
     \norm{\mathcal{A}u}&_{L^{p'}(0,T;V^\ast)} 
\leq \sup\limits_{\stackrel{\varphi \in L^p(0,T;V)}{\norm{\varphi}\leq 1}} \int_0^T \norm{A(t)u(t)}_{V^\ast} 	
	 \norm{\varphi(t)}_V\\
&\leq \sup\limits_{\stackrel{\varphi \in L^p(0,T;V)}{\norm{\varphi}\leq 1}} \int_0^T
	 \big( c_3\norm{u(t)}_V^{p-1}+ c_4 \norm{u(t)}_H^q \norm{u(t)}_V^{p-1}+ C_5(t)\big)  \norm{\varphi(t)}_V\\
&\leq c_3 \norm{u}_{L^p(0,T;V)\cap L^\infty(0,T;H)}^{p-1} +c_4 \norm{u}_{L^p(0,T;V)\cap L^\infty(0,T;H)}^{q+p-1}
	  +\norm{C_5}_{L^{p'}(0,T)} .
\end{align*} 
\end{proof}
From now on we denote $\mathcal{V}:= L^p(0,T;V)$,
$\mathcal{X}:= L^p(0,T;V)\cap L^\infty(0,T;H) $ and \linebreak
$\mathcal{W}:= W^{1,p,p'}(0,T;V,Z^\ast)$.
\begin{lemma}\label{Hirano}
  Let the Assumptions \ref{vorraeume} and \ref{voroperatorfamilie} be
  satisfied. Further assume that
  $v_n \subseteq \mathcal{W}\cap L^\infty(0,T;H) $ satisfies
  $v_n \rightharpoonup v_0$ in $\mathcal{W}$,
  $\norm{v_n}_{L^\infty(0,T;H)} \leq K$ and that
  \begin{align}\label{Hiranovoraussetzung}
    \limsup\limits_{n\rightarrow \infty}\, \langle
    \mathcal{A}v_n,v_n-v_0\rangle_\mathcal{V} \leq 0. 
  \end{align} 
  Then for any $z \in \mathcal{V}$ there holds
  \begin{align}\label{Hiranoresultat}
    \langle\mathcal{A}v_0,v_0 -z\rangle_\mathcal{V}
    \leq\liminf\limits_{n\rightarrow \infty}\, \langle
    \mathcal{A}v_n,v_n-z\rangle_\mathcal{V}. 
  \end{align}
  Moreover, $\mathcal{A}v_n \rightharpoonup\mathcal{A}v_0$ in $\mathcal{V}^\ast= L^{p'}(0,T;V^\ast)$.
\end{lemma}
\begin{proof}
  Fix $z \in \mathcal{V}$. First we choose a subsequence
  $(v_k)_{k\in \Lambda_1}$, $\Lambda_1 \subseteq \mathbb{N}$, such
  that
  \begin{align}\label{hir1}
    \lim\limits_{\stackrel{k\in \Lambda_1}{k\rightarrow \infty}}
    \langle \mathcal{A}v_k,v_k-z\rangle_\mathcal{V}
    =\liminf\limits_{{n\rightarrow \infty}}\,
    \langle \mathcal{A}v_n,v_n-z\rangle_\mathcal{V}. 
  \end{align}
  From (A3) and (A4) in Assumption \ref{voroperatorfamilie} as well as
  Young's inequality we conclude that for a.e. $t \in[0,T]$, all
  $x \in \mathcal{X}$ with $\norm{x}_{L^\infty(0,T;H)} \leq K$ and all
  $y \in \mathcal{V}$ there holds
  \begin{align}\label{hir2}
    \langle A(t)x(t),x(t)-y(t) \rangle_V \geq
    k_1\norm{x(t)}_V^p-k_2\norm{y(t)}_V^p-K_3(t) 
  \end{align}
  with positive constants $k_1,k_2$ (depending on $K$) and a non-negative function
  ${K_3\in L^1((0,T))}$. Especially \eqref{hir2} holds for $x=v_n$ and
  $y=v_0$.
  Since $\mathcal{W} \hookrightarrow C(0,T;Z^\ast)$ holds by
  Proposition \ref{Einbettung1} we get $v_n \rightharpoonup v_0$ in
  $C(0,T;Z^\ast)$. This implies
\begin{align}\label{hir3}
v_n(t)\rightharpoonup v_0(t) \quad \text{in} \ Z^\ast\ \text{for all }
  t \in [0,T].
\end{align} 
Next we show that 
\begin{align}\label{hir4}
\liminf\limits_{\stackrel{k\in \Lambda_1}{k\rightarrow \infty}}\, \langle A(t)v_k(t),v_k(t)-v_0(t) \rangle_V  \geq 0 
\end{align}
holds for a.e. $t \in [0,T]$. To prove this, we define 
\begin{align*}
  B_1&:=\big\lbrace t \in [0,T] \mid 
       \langle A(t)v_k(t),v_k(t)\!-\!v_0(t) \rangle_V \!\geq k_1
       \norm{v_k(t)}_V^p\! -\!k_2 \norm{v_0(t)}_V^P \!-\!K_3(t)
\\
     &\quad\qquad\qquad\qquad\textnormal{holds
  for any } k \in \Lambda_1\big\rbrace,
\\ 
  B_2&:=\big\lbrace t \in [0,T] \mid \norm{v_0(t)}_V,K_3(t)<\infty
       \big\rbrace,
  \\
  B_3&:=\big\lbrace t \in [0,T] \mid A(t) \text{ is pseudo-monotone}\,
       \big\rbrace,
  \\
  B&:= B_1 \cap B_2\cap B_3.
\end{align*} 
Due to $v_0 \in \mathcal{V}$, $K_3 \in L^1((0,T))$, \eqref{hir2} and
(A1) we conclude that $B^\mathsf{c}$ is a set of measure zero. So it is
enough to prove \eqref{hir4} for all $t \in B$. For a fixed $t \in B$
we define $\Lambda_2\subseteq \Lambda_1$ by
\begin{align}\label{hir5}
  \ell \in \Lambda_2 \Leftrightarrow  \langle A(t)v_\ell(t),v_\ell(t)-v_0(t) \rangle_V <0.
\end{align} 
If $\Lambda_2$ is finite, then \eqref{hir4} holds for this fixed
$t \in B$. Thus, assume that $\Lambda_2$ is not finite. Then by the
definition of $B_1$ and \eqref{hir5} we get for all
$\ell \in \Lambda_2$
\begin{align*}
  K_1 \norm{v_\ell(t)}_V^p &\leq \langle
                             A(t)v_\ell(t),v_\ell(t)-v_0(t) \rangle_V
                             +k_2\norm{v_0(t)}_V+K_3(t) 
  \\
                           &\leq k_2\norm{v_0(t)}_V+K_3(t)
\end{align*}
and the right-hand side is finite due to $t \in B_2$. Therefore
$(v_\ell(t))_{\ell\in \Lambda_2}$ is bounded in $V$ and there exists
subsequence $(v_{{\ell}_{j}}(t))_{j\in \mathbb {N}}\subset (v_\ell(t))_{\ell\in
\Lambda_2}$ which converges weakly to $a\in V$.
 % \[
 %  v_{{\ell}_{j}}(t)\stackrel{j\rightarrow\infty}{\rightharpoonup} a
 %  \in V. 
 %  \]
Since $V\hookrightarrow Z^\ast$ we get that $(v_{{\ell}_{j}}(t))_{j\in
  \mathbb {N}} $ also converges weakly to $a\in Z^\ast$. 
% \[
%   v_{{\ell}_{j}} (t)\stackrel{j\rightarrow\infty}{\rightharpoonup} a
%   \in Z^\ast, 
% \]
This together with \eqref{hir3} implies $v_0(t)=a$ in
  $Z^\ast$. Since the embedding $V\hookrightarrow Z^\ast$ is
  injective, we obtain $v_0(t)=a$ in $V.$ This argument is valid for
every weakly convergent subsequence of
$ (v_\ell(t))_{\ell\in \Lambda_2}$ and therefore
\begin{align}\label{hir6}
  v_\ell(t) \rightharpoonup  v_0(t) \quad \text{in} \ V \quad
  (\ell\rightarrow\infty, \ell \in \Lambda_2).
\end{align}
From \eqref{hir5} and \eqref{hir6} together with the
pseudo-monotonicity of $A(t)$ we conclude that
\[
   0\leq \liminf\limits_{\stackrel{\ell\in \Lambda_2}{\ell\rightarrow
    \infty}}\,   \langle A(t)v_\ell(t),v_\ell(t)-v_0(t) \rangle_V.
\]
This proves \eqref{hir4} for every $t \in B$, since for every $k \in
\Lambda_1\setminus \Lambda_2$ we also have 
$\langle A(t)v_k(t),v_k(t)-v_0(t) \rangle_V\geq 0$. 
%  , we get
% \begin{align*}
%   0\leq \liminf\limits_{\stackrel{k\in \Lambda_1}{k\rightarrow \infty}} \langle A(t)v_k(t),v_k(t)-v_0(t) \rangle_V .
% \end{align*}
% So \\

Using {\eqref{hir4}} and Fatou's Lemma we get
\begin{align*}
  0&\leq \int_0^T \liminf\limits_{\stackrel{k\in
     \Lambda_1}{k\rightarrow \infty}}\, \langle A(t)v_k(t),v_k(t)-v_0(t)
     \rangle_V \leq  \liminf\limits_{\stackrel{k\in
     \Lambda_1}{k\rightarrow \infty}}\int_0^T \langle A(t) v_k(t) ,
     v_k(t) -v_0(t) \rangle_V
  \\ 
   &\leq \limsup\limits_{\stackrel{k\in \Lambda_1}{k\rightarrow
     \infty}} \, \langle \mathcal{A} v_k , v_k -v_0 \rangle_\mathcal{V}
     \stackrel{\eqref{Hiranovoraussetzung}}{\leq} 0 ,
\end{align*}
which implies 
\begin{align}\label{hir7}
  \lim\limits_{\stackrel{k\in \Lambda_1}{k\rightarrow \infty}} \langle
  \mathcal{A} v_k , v_k -v_0 \rangle_\mathcal{V} =0. 
\end{align}
We set $h_k(t):= \langle A(t)v_k(t),v_k(t)-v_0(t) \rangle_V$ for
$k\in \Lambda_1$. Thus \eqref{hir4} and \eqref{hir7}
read:
\begin{itemize}
\item[a)] $\liminf\limits_{\stackrel{k\in \Lambda_1}{k\rightarrow \infty}}\, h_k(t)  \geq 0$ for a.e. $t\in I$,
\item[b)] $\lim\limits_{\stackrel{k\in \Lambda_1}{k\rightarrow \infty}} \int_0^T h_k =0$.
\end{itemize}
By \eqref{hir2} and Lebesgue's dominated convergence theorem, we get
${ \lim\limits_{\stackrel{k\in \Lambda_1}{k\rightarrow \infty}}
\int_0^T h_k^-=0}$, with $h_k^-(t):= - \min \lbrace h_k(t),0 \rbrace$. Since
$|h_k|= h_k+2h_k^-$, this implies ${ \lim\limits_{\stackrel{k\in \Lambda_1}{k\rightarrow \infty}}
  \int_0^T \abs{h_k}=0 }$. 
Thus we can choose a subsequence $\Lambda_3 \subseteq \Lambda_1$ such
that
\begin{align}\label{hir8}
  \lim\limits_{\stackrel{j\in \Lambda_3}{j\rightarrow \infty}}\langle A(t)v_j(t),v_j(t)-v_0(t) \rangle_V
  = 0 % \lim\limits_{\stackrel{j\in \Lambda_3}{j\rightarrow \infty}} h_j(t)  =0
\end{align}
holds for a.e.~$t\in [0,T]$. From \eqref{hir2} and \eqref{hir8}
follows that for a.e.~$t \in [0,T]$ the sequence
$(v_j(t))_{j\in\Lambda_3}$ is bounded in $V$ and with the same
argumentation as for \eqref{hir6} we get
\begin{align}\label{hir9a}
  v_j(t) \rightharpoonup v_0(t) \quad
  \text{in} \ V,\quad ( j\rightarrow \infty, \in\Lambda_3 )
\end{align}
for a.e. $t\in [0,T]$. Now \eqref{hir8} and \eqref{hir9a} together
with the pseudo-monotonicity of $A(t)$ imply, that for our fixed
$z \in \mathcal{V}$, there holds
\begin{align}\label{hir10}
  \langle A(t)v_0(t),v_0(t)-z(t) \rangle_V \leq
  \liminf\limits_{\stackrel{j\in \Lambda_3}{j\rightarrow \infty}}\,
  \langle A(t)v_j(t),v_j(t)-z(t) \rangle_V . 
\end{align}
Thus from \eqref{hir1}, \eqref{hir10} and Fatou's lemma we get
\begin{align*}
  \langle \mathcal{A}v_0  ,v_0-z\rangle_\mathcal{V} 
  &\leq \int_0^T \liminf\limits_{\stackrel{j\in
    \Lambda_3}{j\rightarrow \infty}}\, \langle A(t)v_j(t),v_j(t)-z(t)
    \rangle_V
  \\ 
  &\leq \liminf\limits_{\stackrel{j\in \Lambda_3}{j\rightarrow
    \infty}}\, \langle \mathcal{A}v_j
    ,v_j-z\rangle_\mathcal{V}=\lim\limits_{\stackrel{k\in
    \Lambda_1}{k\rightarrow \infty}} \langle \mathcal{A}v_k
    ,v_k-z\rangle_\mathcal{V}
  \\ 
  &=\liminf\limits_{n\rightarrow \infty}\,
    \langle \mathcal{A}v_n,v_n-z\rangle_\mathcal{V}. 
\end{align*}
Therefore we proved \eqref{Hiranoresultat} for any $z\in
\mathcal{V}$. This together with \eqref{Hiranovoraussetzung} implies
\begin{align}\begin{split}\label{hir11}
    \langle\mathcal{A}v_0,v_0-z\rangle_\mathcal{V} &\leq
    \liminf\limits_{n\rightarrow \infty}\, \langle\mathcal{A}v_n,v_n
    -z\rangle_\mathcal{V}
    \\
    &\leq \limsup\limits_{n\rightarrow
      \infty}\, \langle\mathcal{A}v_n,v_n-v_0
    \rangle_\mathcal{V} + \liminf\limits_{n\rightarrow \infty} \,
    \langle\mathcal{A}v_n,v_0 -z\rangle_\mathcal{V}    
    \\
    &\leq \liminf\limits_{n\rightarrow \infty}\,
    \langle\mathcal{A}v_n,v_0 -z\rangle_\mathcal{V}
\end{split}\end{align}
for any $z\in \mathcal{V}$. Using that we can replace $z$ by $2v_0-z$
in \eqref{hir11}, we obtain in a standard manner 
% and multiply the inequality with $(-1)$ we get 
% \begin{align}\label{hir12}
% \langle \mathcal{A}v_0, v_0 -z  \rangle_\mathcal{V} \geq - \liminf\limits_{n\rightarrow \infty}  \langle\mathcal{A}v_n,z-v_0\rangle_\mathcal{V} = \limsup\limits_{n\rightarrow \infty}  \langle\mathcal{A}v_n,v_0-z\rangle_\mathcal{V}.
% \end{align}
% Now \eqref{hir11} and \eqref{hir12} yield
\begin{align*}
  \langle \mathcal{A}v_0, v_0 -z  \rangle_\mathcal{V}\leq
  \liminf\limits_{n\rightarrow \infty}\, \langle\mathcal{A}v_n,v_0
  -z\rangle_\mathcal{V} \leq \limsup\limits_{n\rightarrow \infty}\,
  \langle\mathcal{A}v_n,v_0 -z\rangle_\mathcal{V} \leq \langle
  \mathcal{A}v_0, v_0 -z\rangle_\mathcal{V}. 
\end{align*}
for any $z\in \mathcal{V}$, i.e. $\mathcal{A}v_n \rightharpoonup \mathcal{A}v_0$ in $\mathcal{V}^\ast$.
\end{proof}

\section{Proof of the main theorem}
Let $\lbrace v_n^k\rbrace_{k=1}^{d_n}$ be a basis of $V_n$,
$d_n=\operatorname{dim}V_n$. We are seeking approximative solutions
\begin{align*}
u_n(t) = \sum\limits_{k=1}^{d_n} c_n^k(t) v_n^k,
\end{align*}
which solve the Galerkin system
\begin{align}
  \begin{split}\label{galerkingl}
    \left\langle \frac{d u_n(t)}{dt}, v_n^k \right\rangle_V
    +\left\langle A(t)u_n(t), v_n^k \right\rangle_V
    &= \left\langle f(t), v_n^k  \right\rangle_V  \qquad
    k=1,\ldots,d_n
    \\
    u_n(0)&=u_{0,n},
  \end{split}
\end{align}
where $u_{0,n} \in V_n$ are chosen such that $u_{0,n}
\stackrel{n\rightarrow \infty}{\longrightarrow} u_0$ in $H$.

Since the matrix
$D^n= \big( ( v_n^i, v_n^j ) \big)_{i,j=1,\ldots,d_n}$ is positive
definite, the Galerkin system \eqref{galerkingl} can be re-written as a
system of ordinary differential equations. From (A1) and (A4) we get
that $A(t)$ is a pseudo-monotone and bounded operator, which yields
that $A(t)$ is demicontinuous from $V\rightarrow V^\ast$ \cite[Lemma
2.4]{Roub}.  This implies that the system of ordinary differential
equations fulfills the Carath\'{e}odory conditions and is therefore
solvable locally in time, i.e.~on a small time interval $[0,t_0]$.  We
multiply the Galerkin system \eqref{galerkingl} with $c_n^k(t)$, sum
from $1$ to $d_n$,  integrate over $[0,t_0]$ and
use (A3) as well as Young's inequality to obtain
\begin{align}
  \begin{split}\label{apriori1}
    \frac{1}{2} &\norm{u_n(t)}_H^2
    +(c_1-\varepsilon)\int_0^{t_0} \norm{u_n(t)}_V^p 
    \\ 
    &\leq \frac{1}{2} \norm{u_n(0)}_H^2 + c(\varepsilon) \int_0^{t_0}
    \norm{f(t)}_{V^\ast} + \int_0^{t_0} C_2(t) \leq K(u_0,f,C_2).
  \end{split}
\end{align}
This proves the boundedness of $\abs{c_n^k(t)}$ on $[0,t_0]$ for any
$1\leq k \leq d_n$ and we can iterate Carath\'{e}odory's
theorem to get the existence of an absolutely continuous solution
$u_n$ on $[0,T]$, which satisfies the a~priori estimates
\begin{align}
  \begin{split}\label{apriori2}
    \norm{u_n}_{L^\infty(0,T;H)} + \norm{u_n}_{L^p(0,T;V)} &\leq K,
    \\
    \norm{\mathcal{A}u_n}_{L^{p'}(0,T;V)^\ast} &\leq K,
  \end{split}
\end{align}
where we used Lemma \ref{induzierterOp} for the second estimate.

To get an uniform estimate for the time derivative we proceed as
follows. For any $v\in L^p(0,T;Z)$ we have $P_nv\in L^p(0,T;V_n)$,
which therefore is an admissible test-function in 
\eqref{galerkingl}. Since $P_n$ is self-adjoint and
$P_{n\ |V_n}=\operatorname{Id}$, we  conclude
\begin{align*}
  \begin{split}
    \left\langle \frac{d u_n}{dt},v \right\rangle_{L^p(0,T;Z)}
    &=\left\langle P_n\frac{d u_n}{dt},v \right\rangle_{L^p(0,T;Z)}
    =\left\langle \frac{d u_n}{dt},P_nv \right\rangle_{L^p(0,T;Z)} 
    \\ 
    &= \left\langle f,P_nv \right\rangle_{L^p(0,T;Z)} -\left\langle
      \mathcal{A}u_n,P_nv \right\rangle_{L^p(0,T;Z)}  
    \\
    &\leq c \left( \norm{\mathcal{A}u_n}_{L^{p'}(0,T;V)^\ast} +
      \norm{f}_{L^{p'}(0,T;V)^\ast}  \right)
    \norm{P_nv}_{L^p(0,T;V)}
    \\ 
    &\leq c \norm{v}_{L^p(0,T;Z)}.
  \end{split}
\end{align*}
This yields
\begin{align}
  \begin{split}\label{zeitableitung}
    \norm{ \frac{du_n}{dt}  }_{L^{p'}(0,T;Z^\ast)} \leq K.
  \end{split}
\end{align}
The estimates \eqref{apriori2} and \eqref{zeitableitung} yield
\begin{alignat}{2}\raisetag{2cm}\label{konv}
  \begin{aligned}
    u_n &\rightharpoonup u& \quad &\text{in} \ L^{p}(0,T;V),
    \\
    u_n &\overset{\ast}{\rightharpoonup} u &&\text{in} \
    L^{\infty}(0,T;H),
    \\
    \mathcal{A}u_n &\rightharpoonup \zeta &&\text{in} \
    L^{p'}(0,T;V^\ast),
    \\
    \frac{du_n}{dt} &\rightharpoonup \frac{du}{dt} &&\text{in} \
    L^{p'}(0,T;Z^\ast). 
  \end{aligned}
\end{alignat}
From this, the embedding $W^{1,p,p'}(0,T;V,Z^\ast) \hookrightarrow
C(0,T;Z^\ast)$ (cf.~Proposition \ref{Einbettung1}), the injectivity of
the embedding $H \hookrightarrow Z^\ast$, the convergence $u_{0,n}\to
u_0$ in $H$ and \eqref{apriori2}$_1$ we obtain
\begin{alignat}{2}\raisetag{2cm}\label{konv1}
  \begin{aligned}
    u_n(0) &\rightarrow u(0)=u_0& \quad &\text{in} \ H,
    \\
    u_n(T) &\rightharpoonup u(T)& \quad &\text{in} \ H.
  \end{aligned}
\end{alignat}
Let $v \in V_k$, $k \in {\mathbb N}$, and $\varphi \in C_0^\infty((0,T))$. If we use
integration by parts for real-valued functions, \eqref{galerkingl}
reads as 
\begin{align*}\begin{split}
    -\int_0^T (u_n(t),v)_H \varphi'(t) = \int_0^T \left\langle
      f(t)-A(t)u_n(t),v \right\rangle_V \varphi(t) 
\end{split}\end{align*} 
and \eqref{konv} allows us to pass to the limit in every term,
yielding 
\begin{align}\begin{split}\label{GWgleichung}
    -\int_0^T (u(t),v)_H \varphi'(t) = \int_0^T \left\langle
      f(t)-\zeta(t),v \right\rangle_V \varphi(t). 
\end{split}\end{align}
This and the density of $\cup_{k\in\mathbb{N}}V_k$ in $V$ yields that
\eqref{GWgleichung} is valid for every $v\in V$ and $C_0^\infty((0,T))$. 
Now Proposition \ref{ExistenzZeitab} implies
\begin{align}
  \begin{split}\label{gleichungdu}
    \frac{du}{dt}+\zeta=f \qquad \text{in} \ L^{p'}(0,T;V^\ast)
  \end{split}
\end{align}
and therefore $u \in W_p^1(0,T;V,H) \hookrightarrow C(0,T;H)$.

We want to use Lemma \ref{Hirano} to prove that
$\mathcal{A}u=\zeta$. From \eqref{konv} we know that
$(u_n) \subseteq \mathcal{W}\cap L^\infty(0,T;H) $ satisfies
$u_n \rightharpoonup u_0$ in $\mathcal{W}$,
$\norm{u_n}_{L^\infty(0,T;H)} \leq K$ so we only need to check that
\eqref{Hiranovoraussetzung} holds. To this end we test
\eqref{galerkingl} with $u_n$ and use integration by parts,
\eqref{konv}, \eqref{konv1} as well as the weak lower semicontinuity
of the norm to conclude
\begin{align*}
  \begin{split}
    \limsup\limits_{n\rightarrow \infty}\, \langle \mathcal{A}u_n,u_n
    \rangle_\mathcal{V} \leq \langle f,u\rangle_\mathcal{V}
    -\frac{1}{2}\norm{u(T)}_H^2  +\frac{1}{2}\norm{u_0}_H^2. 
  \end{split}
\end{align*}
If we test \eqref{gleichungdu} with $u$, use the integration by
parts formula from Proposition \ref{Einbettung2} and \eqref{konv1}, we obtain  
\begin{align*}
 \langle \zeta,u \rangle_\mathcal{V} = \langle f,u\rangle_\mathcal{V}
  -\frac{1}{2}\norm{u(T)}_H^2  +\frac{1}{2}\norm{u_0}_H^2 .  
\end{align*}
The last two equalities imply \eqref{Hiranovoraussetzung} and Lemma
\ref{Hirano} yields $\mathcal{A}u=\zeta$, so that
\begin{align}
  \begin{split}
    \frac{du}{dt}+\mathcal{A}u=f \qquad \text{in} \ L^{p'}(0,T;V^\ast).
  \end{split}
\end{align}

\section{Examples}
In this section we illustrate the above theory by two
applications. Let us start with some notation.  Let
$\Omega_1 \subseteq \mathbb{R}^3$ and
$\Omega_2 \subseteq \mathbb{R}^d$ be bounded domains with
Lipschitz-boundary. Let
$V_1:= W_{0,\operatorname{div}}^{1,p}(\Omega_1)^3=$ and
$H_1:= L_{\operatorname{div}}^2(\Omega_1)^3$ be the completion of
$\lbrace \varphi\in C_0^\infty(\Omega_1)^3 \mid \operatorname{div}
\varphi =0\rbrace$
in $W^{1,p}$ and $L^2$ respectively. Additionally we define
$V_2:=W_0^{1,p}(\Omega_2)$ and $H_2:=L^2(\Omega_2)$. We assume further
that a function
$g:[0,T]\times\Omega_2\times \mathbb{R}\subseteq \mathbb{R} \times
\mathbb{R}^n \times \mathbb{R} \rightarrow \mathbb{R}$
satisfies:
\begin{enumerate}
\item[(g1)] g is measurable in its first two variables and continuous in its third variable.
\item[(g2)] For some constant $c_6>0$, some non-negative function $C_7 \in L^{p'}((0,T))$ and $1\leq r \leq p\frac{d+2}{d}$ holds
  \[
    \abs{g(t,x,s)} \leq c_6(1+\abs{s}^{r-1})+C_7(t).
  \] 
\item[(g3)] For some non-negative $C_8\in L^1(0,T)$ there holds
\[g(t,x,s)\cdot s \geq - C_8(t).\]  
\end{enumerate}
Then we have the following two existence theorems.
\begin{theorem}\label{plaplace}
  For $0<T<\infty$ we set $\Omega_1^T:=[0,T] \times \Omega_1$.
  If $\frac{11}{5}\leq p $, then for any
  $f \in L^{p'}(0,T;V_1^\ast)$ and $u_0 \in H_1$ there exists a
  function $u \in W_p^1(0,T;V_1,H_1)$ with $u(0)=u_0$ in $H_1$, which
  satisfies the equation
  \begin{align*}
    \int_0^T \!\!\left\langle \frac{du(t)}{dt}
    ,\varphi(t)\right\rangle_{V_1} \!+\! \int_{\Omega_1^T} \abs{\nabla
    u}^{p-2} \nabla u :\nabla \varphi - u\otimes
    u :\nabla \varphi = \int_0^T \!\!\left\langle f(t)
    ,\varphi(t)\right\rangle_{V_1}
  \end{align*}
  for any $\varphi \in L^p(0,T;V_1)$.
\end{theorem}
\begin{theorem}\label{kompaktestoerung}
  For $0<T<\infty$ we set $\Omega_2^T:=[0,T] \times \Omega_2$.  Assume
  that the function $g$ satisfies the properties (g1), (g2) and
  (g3). If Let $\frac{2d}{d+2}<p$, then for every $u_0 \in H_2$ and
  $f \in L^{p'}(0,T;(W_0^{1,p})^\ast)$ there exists a function
  $u \in W_p^1(0,T;V_2,H_2)$ with $u(0)=u_0$, such that for any
  \mbox{$\varphi \in L^p(0,T;W_0^{1,p})$} 
  \begin{align*}
    \int_0^T \!\!\left\langle \frac{du(t)}{dt}
    ,\varphi(t)\right\rangle_{V_2}\! + \!\int_{\Omega_2^T} \abs{\nabla
    u}^{p-2} \nabla u \cdot \nabla \varphi +
    g(\boldsymbol{\cdot},\boldsymbol{\cdot},u) \cdot \varphi =
    \int_0^T \!\!\left\langle f(t) ,\varphi(t)\right\rangle_{V_2}.
  \end{align*}
\end{theorem}

\begin{lemma}\label{bspraeume}
  $(V_1,H_1,(V_1)^\ast)$ and $(V_2,H_2,(V_2)^\ast)$, resp., are
  Gelfand-triples, which satisfy Assumption \ref{vorraeume} with
  $Z_1:= \overline{\lbrace \varphi\in C_0^\infty(\Omega_1)^3 \mid
    \operatorname{div} \varphi =0\rbrace}^{W^{s_1,2}}$
  and \linebreak ${Z_2:= W_0^{s_2,2}(\Omega_2)}$, respectively.
\end{lemma}
\begin{proof}
  From Sobolev's embedding theorem we conclude that
  $(V_1,H_1,(V_1)^\ast)$ and $(V_2,H_2,(V_2)^\ast)$ are
  Gelfand-triples. It is also clear that $Z_1$ and $Z_2$, resp., are
  densely embedded into $V_1$ and $V_2$, resp., if $s_1$ and
  $s_2$ are chosen appropriately. A proof for the existence of 
  projections $P_n$ can be found e.g.~in the appendix of \cite{MNRR}.
%>>>>>>> rose/master
\end{proof}
Next we want to introduce some operators for the proofs of our
theorems. We define $B_1, B_2, A_1:V_1 \rightarrow (V_1)^\ast$ and
$B_3, B_4(t), A_2(t):V_2 \rightarrow (V_2)^\ast$ by
\begin{align*}
  \langle B_1u ,v \rangle_{V_1}
  &:= \int_{\Omega} \abs{\nabla u}^{p-2}
                                 \nabla u :\nabla v
  \\
  \langle B_2u ,v \rangle_{V_1}&:= -\int_{\Omega} u\otimes u: \nabla v
  \\
  \langle B_3u ,v \rangle_{V_2}&:= \int_{\Omega} \abs{\nabla u}^{p-2}
                                 \nabla u \cdot\nabla v
  \\
  \langle B_4(t)u ,v \rangle_{V_2}&:= \int_{\Omega}
                                    g(t,\boldsymbol{\cdot},u)\cdot v
\end{align*}
The operators $B_1$ and $B_3$ are the well-known $p$-Laplace
operators. The basic properties of these operators can be found e.g.~\cite[Theorem
17.11]{Ben}.
\begin{theorem}\label{resultatplaplace}
  Let $\Omega\subseteq \mathbb{R}^d$ be a bounded domain,
  $p \in (1,\infty)$. Then the $p$-Laplace operator maps
  $W_0^{1,p}(\Omega)$ into $W_0^{1,p}(\Omega)^\ast$ and is monotone
  and continuous.
\end{theorem}

Note that one can replace the $p$-Laplace operators by operators $B$
having a so-called $(p,\delta)$-structure, with $\delta \in [0,
\infty)$, $p\in (1,\infty)$, as e.g.
\begin{align*}
  \langle Bu ,v \rangle_{W_0^{1,p}}
  &:= \int_{\Omega} \big (\delta+\abs{\nabla u}\big )^{p-2}
    \nabla u :\nabla v.
 \end{align*}

\begin{proof}[Proof (of Theorem \ref{plaplace})]
  From Lemma \ref{bspraeume} we know that Assumption \ref{vorraeume}
  is satisfied. So we only have to check that $ A_1:= B_1 +B_2$
  satisfies Assumption \ref{voroperatorfamilie}. We restrict ourselves
  to the more complicated case $p <3$. For $p\ge 3$ one has to use
  different embedding theorems and some calculation simplify.

  (A1) Due to Proposition \ref{pseudo} and Theorem
  \ref{resultatplaplace} we only have to prove that
  $B_2:V_1 \rightarrow (V_1)^\ast$ is strongly continuous. Let
  $u_n\rightharpoonup u$ in $V_1$. By the Sobolev embedding theorem we
  get $u_n \rightarrow u$ in $L^s(\Omega_1)$ for any
  $1\leq s <\frac{3p}{3-p}$. Then H\"{o}lder's inequality yields
  \begin{align*}
    \sup\limits_{\norm{v}_{V_1}\leq 1}
    &\abs{\langle Au_n -Au, v
      \rangle_{V^p}}
    \\
    &\leq\sup\limits_{\norm{v}_{V_1}\leq 1} \big( \norm{u_n
      -u}_{L^s}\norm{u}_{L^s} \norm{\nabla v}_{L^p} +\norm{u_n
      }_{L^s}\norm{u_n-u}_{L^s} \norm{\nabla v}_{L^p} \big)
    \\ 
    &\stackrel{n\rightarrow \infty}{\longrightarrow} 0.
  \end{align*}
  
  (A2) Fix an arbitrary $u \in L^p(0,T;V_1)\cap
  L^\infty(0,T;H_1)$.
  Due to Pettis' Theorem we only have to prove that
  $t \mapsto \langle A_1u(t),v \rangle_{V_1}$ is Lebesgue-measurable
  for any $v \in V_1$. The functions
  $\abs{\nabla u(t,x)}^{p-2} \nabla u(t,x):\nabla v(x)$ and
  $u(t,x)\otimes u(t,x):\nabla v(x)$ are Lebesgue-measurable on
  $[0,T]\times \Omega_1$ and we easily estimate
  \begin{align*}
    &\int_0^T \int_{\Omega_1} \abs{\nabla u(t,x)}^{p-1} \abs{\nabla
      v(x)} \leq  \int_0^T \norm{\nabla u(t)}_{L^p}^{p-1}\norm{\nabla
      v}_{L^p}\leq c \norm{u}_{L^p(0,T;V_1)}^{p-1}\norm{v}_{V_1},
    \\ 
    &\int_0^T\int_{\Omega_1} \abs{ u(t,x)\otimes u(t,x):\nabla v(x)
      } \leq c \norm{u}^2_{L^{\frac{5}{3}p}([0,T]\times\Omega)}
      \norm{v}_{V_1}. 
  \end{align*}
  Since
  $L^p(I;V_1)\cap L^\infty(I;H_1) \hookrightarrow
  L^{\frac{5}{3}p}(I\times\Omega_1)$,
  see \cite[Proposition 3.1]{DiBen}, we get form Fubini's Theorem that
  \[
    t\mapsto \langle A_1u(t),v \rangle_{V_1} = \int_{\Omega_1}
    \abs{\nabla u(t,x)}^{p-2} \nabla(t,x) :\nabla v(x)
    -\int_{\Omega_1} u(t,x)\otimes u(t,x):\nabla v(x) 
  \]
  is Lebesgue-measurable.

  (A3) For any $u \in V_1$ there holds
  \begin{align*}
    \langle A_1u ,u \rangle_{V_1}= \int_{\Omega_1} \abs{\nabla u}^p
    -\int_{\Omega_1} u\otimes u :\nabla u  =\norm{u}_{V_1}^p, 
  \end{align*}
  since the second term vanishes due to $\operatorname{div} u=0$.

  (A4) From H\"{o}lder's inequaltiy we estimate
  \begin{align*}
    \norm{B_1u}_{(V_1)^\ast}  &\leq \norm{\nabla u}_{L^p}^{p-1}= \norm{u}_{V_1}^{p-1},\\
    \norm{B_2u}_{(V_1)^\ast} &\leq \norm{u}_{L^{2p'}}^2.
  \end{align*}
  For $r=\frac{12p}{-5p^2+17p-6}$ we use the H\"{o}lder interpolation
  \[
    \norm{u}_{L^{r}} \leq
    \norm{u}_{L^{2}}^{\frac{3-p}{2}}\norm{u}_{L^{\frac{3p}{3-p}}}^{\frac{p-1}{2}}.
  \]
  Since for any $p \in [\frac{11}{5},3)$ there holds that $2p'\leq r$
  we conclude that
  \begin{align*}
    \norm{B_2u}_{(V_1)^\ast} \leq C
    \norm{u}_{L^{2}}^{{3-p}}\norm{u}_{L^{\frac{3p}{3-p}}}^{p-1}\leq  C
    \norm{u}_{H_1}^{{3-p}}\norm{u}_{V_1}^{p-1}. 
  \end{align*}
  So all the assumptions are satisfied and Theorem \ref{maintheorem}
  proves Theorem \ref{plaplace}.
%>>>>>>> rose/master
\end{proof}

\begin{proof}[Proof (of Theorem \ref{kompaktestoerung})]
  This works basically the same way as the proof of Theorem
  \ref{plaplace}. The assumption \ref{vorraeume} is fulfilled by Lemma
  \ref{bspraeume}, so that we only have to check Assumption
  \ref{voroperatorfamilie} for $A_2(t):=B_3 +B_4(t)$ to apply Theorem
  \ref{maintheorem}. Again, we restrict ourselves
  to the more complicated case $p <d$. For $p\ge d$ one has to use
  different embedding theorems and some calculation simplify.

  (A1) Again we only have to prove the pseudo-monotonicity of $B_4(t)$
  for a.e. $t \in [0,T]$. Using (g2) and H\"{o}lder's inequality we
  easily estimate
  \begin{align*}
    \abs{\langle B_4(t)u,\varphi\rangle_{V_2} }
    &\leq c_6(1+C_7(t)
      +\norm{u}_{L^{(r-1)\sigma'}}^{r-1})\norm{\varphi}_{L^\sigma}
    \\
    &\leq c_6(1+C_7(t) +\norm{u}_{V_2}^{r-1})\norm{\varphi}_{V_2},
  \end{align*}
  where $\sigma:= \frac{dp}{d-p}$ is the Sobolev exponent and
  $\frac{1}{\sigma}+\frac{1}{\sigma'}=1$. The last inequality holds if
  $(r-1)\sigma'\leq \sigma$ but this is equivalent to $r\leq
  \sigma$.
  For any $p\in (\frac{2d}{d+2},d)$ there holds
  $\sigma=\frac{dp}{d-p}> p\frac{d+2}{d}\ge r$, due to (g2).  
  Thus the Sobolev embedding theorem implies $u_n \rightarrow u$ in
  $L^r(\Omega_2)$ for $u_n \rightharpoonup u$ in $V_2$.  Now define
  $F_t:L^r(\Omega_2) \rightarrow L^{r'}(\Omega_2)$ by
  $F_t(u)(x):=g(t,x,u(x))$. Condition (g2) and the theory of
  Nemyckii-Operators, cf. \cite[Theorem 1.43]{Roub}, implies the
  continuity of $F_t$ and 
  \begin{align*}
    \sup\limits_{\substack{\varphi \in V_2\\ \norm{\varphi}_{V_2}\leq
    1}} \abs{\langle A_2(t)u_n - A_2(t)u,\varphi\rangle_{V_2}  }\leq
    \sup\limits_{\substack{\varphi \in V_2\\ \norm{\varphi}_{V_2}\leq
    1}} \norm{F_t(u_n)-F_t(u)}_{L^{r'}}\norm{\varphi}_{V_2}
    \rightarrow 0. 
  \end{align*}
  So $B_4(t)$ is strongly continuous and therefore pseudo-monotone.

  (A2) Using Pettis' Theorem we only have to check that
  \[
    t\mapsto \langle A_2(t)u(t),\varphi \rangle_{V_2} =
    \int_{\Omega_2} \!\abs{\nabla u(t,x)}^{p-2} \nabla u(t,x)\cdot
    \nabla v(x) + g(t,x,u(t,x))\cdot v(x)
  \] 
  is Lebesgue-measurable for arbitrary $v$ in $V_2$. The functions
  $g(t,x,u(t,x))\cdot v(x)$ and
  $\abs{\nabla u(t,x)}^{p-2} \nabla u(t,x)\cdot \nabla v(x)$ are
  Lebesgue- measurable on $[0,T]\times \Omega_2$. Using the growth
  condition (g2) we estimate
  \begin{align*}
    &\int_0^T \int_{\Omega_2} \abs{\nabla u(t,x)}^{p-1} \abs{\nabla
      v(x)} \leq  \int_0^T \norm{\nabla
      u(t)}_{L^p}^{p-1}\norm{\nabla v}_{L^p}\leq
      T^{\frac{1}{p}}\norm{u}_{L^p(0,T;V_2)}\norm{v}_{V_2},
    \\ 
    &\int_0^T\int_{\Omega_2} g(t,x,u(t,x))\cdot v(x)\leq
      c_6(1+\norm{C_7}_{L^{p'}} +
      \norm{u}^{r-1}_{L^p(I;W_0^{1,p})\cap L^\infty(I;L^2)}
      )\norm{v}_{W_0^{1,p}}, 
  \end{align*}
  so that Fubini's Theorem again proves the assertion.
    
  (A3) Using (g3) we have
  \begin{align*}
    \langle A_2(t)u ,u \rangle_{W_0^{1,p}}
    &= \int_{\Omega} \abs{\nabla u(x)}^p + g(t,x,u(x))\cdot u(x) 
    \\
    &\geq \norm{u}_{V_2}^p -C_8(t).
  \end{align*}

  (A4) Using H\"{o}lder's inequality we easily get
  \begin{align}\label{g6.4}
    \norm{B_3u}_{(V_2)^\ast}  \leq\norm{u}_{V_2}^{p-1}.
  \end{align}
  From (g2) we estimate for $u$ in $V_2$
  \begin{align}
    \begin{split}\label{ersteabsch}
      \norm{B_4(t)u}_{(V_2)\ast} &\leq c_6(1+C_7(t)) + c
      \norm{u}_{L^{\sigma'(r-1)}}^{r-1}
      \\
      &\leq c_6(1+C_7(t)) + c \norm{u}_{L^{\sigma'(r_0-1)}}^{r_0-1}
    \end{split}
  \end{align}
  where $r_0 := p\frac{d+2}{d}$ and $\sigma=\frac{dp}{d-p}$. Using the
  interpolation
  \begin{align*}
    \norm{u}_{L^{(r_0-1)\sigma'}}\leq \norm{u}_{L^2}^{1-\lambda} \norm{u}_{L^{\frac{dp}{d-p}}}^{\lambda}
  \end{align*}
  we see that $\lambda(r_0-1)\le p-1$. This and the embedding $V_2
  \hookrightarrow L^\sigma$ yields
  \begin{align*}
    \norm{A_2(t)u}_{(V_2)^\ast} \leq c_6(1+C_7(t)) + c \norm{u}_{L^2}^{(r_0-1)(1-\lambda)} \norm{u}_{L^{\frac{dp}{d-p}}}^{p-1}\\
    \leq c_6(1+C_7(t)) + c \norm{u}_{H_2}^{(r_0-1)(1-\lambda)} \norm{u}_{V_2}^{p-1}.
  \end{align*}
  So Assumption \ref{voroperatorfamilie} is satisfied.
%>>>>>>> rose/master
\end{proof}

% of $L^{\sigma'(r_0-1)}$ between $L^2$ and $L^\sigma$ 

% Setting
% \begin{align*}
% \sigma&:= \frac{dp}{d-p}\\
% \alpha&:= (r_0 -1)p'\\
% \beta &:= \frac{(r_0-1)pd}{d(p-1)+p}=(r_0-1)\sigma'.
% \end{align*}
% it is easy to show that the subsequent inequalities are valid.
% \begin{enumerate}
% \item $\alpha>p$
% \item $\beta<\sigma$
% \item $\beta>2$
% \item $\frac{\sigma}{\beta} +\frac{p}{\alpha}\big( \frac{\sigma}{2}  -1\big)\geq \frac{\sigma}{2}$
% \end{enumerate}
% This allows us to apply a parabolic interpolation result, cf. \cite{}, to estimate
% Now \eqref{g6.4}, \eqref{ersteabsch} and \eqref{interpol2} imply

\end{document}